\documentclass[10pt]{amsart}
\usepackage[utf8]{inputenc}
\usepackage{lmodern}
\usepackage[english]{babel}

\usepackage{amsthm}
\usepackage{amssymb}
\usepackage{mathtools}
\usepackage{graphicx}
\usepackage{amscd}
\usepackage[hidelinks,unicode]{hyperref}

\theoremstyle{plain}
\newtheorem{theorem}{Theorem}[section]
\newtheorem{lemma}[theorem]{Lemma}
\newtheorem{corollary}[theorem]{Corollary}
\newtheorem{proposition}[theorem]{Proposition}

\theoremstyle{definition}
\newtheorem{definition}{Definition}[section]

\theoremstyle{remark}
\newtheorem*{remark}{Remark}

\newcommand{\WR}{\mathbin{\mathrm{wr}}}
\newcommand{\bbz}{\mathbb{Z}}
\newcommand{\Aut}{\mathop{\mathrm{Aut}}}

\newcommand{\Id}{\mathop{\mathrm{Id}}}

\emergencystretch=\maxdimen
\begin{document}

	\title[TBFT and $R_\infty$-Property for Lamplighter-Type Groups]{Twisted Burnside--Frobenius Theorem and $R_\infty$-Property for Lamplighter-Type Groups}

	\author{Mikhail I.~Fraiman}

\begin{abstract}
	We prove that the restricted wreath product ${\bbz_n \WR \bbz^k}$ has the $R_\infty$-property,
	i.~e.\ every its automorphism~$\varphi$ has infinite Reidemeister number~$R(\varphi)$,
	in exactly two cases: (1)~for any $k$ and even $n$; (2)~for odd $k$ and $n$ divisible by~3.
	
	In the remaining cases there are automorphisms with finite Reidemeister number, for which we prove the finite-dimensional twisted Burnside--Frobenius theorem (TBFT\textsubscript{$f$}):
	$R(\varphi)$ is equal to the number of equivalence classes of finite-dimensional irreducible unitary representations fixed by the action ${[\rho]\mapsto[\rho\circ\varphi]}$.

\end{abstract}

\maketitle

	\section{Introduction}
		
		Let $\varphi$ be an automorphism of a group $G$.
		Two elements ${g_1,g_2\in G}$ are said to be \emph{$\varphi$-twisted conjugate} if there exists
			${h\in G}$, such that ${g_1=hg_2\varphi(h^{-1})}$.
		It is well-know that this relation is an equivalence relation.
		The $\varphi$-class of an element~$g$ is denoted by~${\{g\}_\varphi}$ and is called the \emph{Reidemeister class} of the  element~$g$.
		The number of such classes is called the \emph{Reidemeister number of automorphism~$\varphi$} and is denoted by~$R(\varphi)$.
		A group is said to have \emph{$R_\infty$-property} if every its automorphism has infinite Reidemeister number.
		
		The study of Reidemeister classes is motivated by the study of Nielsen classes, which play an important role in topological fixed point theory.
		For a finite cell complex $X$ and a continuous map ${f \colon X \to X}$ the number $N(f)$ of $f$-Nielsen classes is a homotopy invariant, but is difficult for calculation.
		In some cases, however, this number coincides with the Reidemeister number ${R(\varphi)}$, where ${\varphi = f_\#\in\Aut\pi_1(X)}$.
		This transforms the topological problem into an algebraic one, namely, studying of the structure of Reidemeister classes.
		A study of algorithm of determining whether two points have the same Reidemeister class was conducted in~\cite{RCFG} for finitely generated free groups.
		
		Another important problem is the study of $R_\infty$-property for groups.
		Recent results in this field can be found, in particular, in~\cite{RCSWBG,RCNG,RCUE,RCEGG,RCFGFNFS,RCAR}.
		In~\cite{TR} it was shown that ${\bbz_p\WR\bbz^k}$ has the $R_\infty$-property for prime $p$ in exactly two cases: ${p=2}$ and ${p=3, k=2d+1}$ for some natural $d$.
		
		\textbf{Paper Structure.} Section~\ref{MF:secPRE} serves as a brief overview of some general facts related to Reidemeister classes.
		Also, the form of automorphisms of a wreath product is discussed.
		Then, we introduce three main results of the paper:
			in Section~\ref{MF:secPT} we prove Theorem~\ref{MF:theor1}, which implies $R_\infty$-property for the groups ${\bbz_{2n} \WR \bbz^k}$ and ${\bbz_{3n}\WR\bbz^{2k+1}}$;
			in Section~\ref{MF:secRI} we prove Theorem~\ref{MF:theor2}, which has two applications --- combined with Propositions~\ref{MF:prop3} and~\ref{MF:prop2d} it implies lack of $R_\infty$ for all the remaining groups among ${\bbz_n \WR \bbz^k}$, and combined with Proposition~\ref{MF:tbftfor} implies TBFT\textsubscript{$f$} for automorphisms with finite Reidemeister number (Theorem~\ref{MF:theortbft}).

		\textbf{Acknowledgements.} The author is thankful to E.~V.~Troitsky for suggesting the problem and for useful advices related to the research.
		
		The work was supported by the Foundation for the Advancement of Theoretical Physics and Mathematics ‘BASIS’ under research grant 19-7-1-34-2.

	\section{Preliminaries}\label{MF:secPRE}
	
	\subsection{General theory.} We begin with an overview of general facts, related to the Reidemeister classes and wreath products.
		
		\begin{proposition}[see~\cite{RCB}]
			For an abelian group the $\varphi$-twisted conjugacy class of the unit element is a subgroup and other classes are the corresponding cosets.
		\end{proposition}

		\begin{remark}
			In case of abelian group the $\varphi$-class of the unit element coincides with ${\mathrm{Im}(1-\varphi)}$ and the index of this subgroup equals $R(\varphi)$.
		\end{remark}

		\begin{proposition}[see~\cite{TR}]
			Let $\tau_g$ be the inner automorphism of a group~$G$, i.~e. ${ \tau_g(h) = ghg^{-1} }$.
			Then the right shift action by the element~$g$ maps Reidemeister classes of~$\varphi$ onto Reidemeister classes of~${\tau_{g^{-1}}}\circ \varphi$.
			In particular,~${R(\varphi) = R(\tau_g\circ\varphi)}$.
		\end{proposition}

		The following proposition can be found in~\cite{TR} and will be very useful in the future:

		\begin{proposition}\label{MF:prop23}
			Let~$G$ be a group and $\varphi$ be its automorphism.
			Let $H$ be a $\varphi$-invariant normal subgroup.
			Denote by $C(\varphi)$ the fixed-point subgroup, by~$\varphi'$ the restriction of~$\varphi$ on~$H$, by $\overline{\varphi}$ the quotient automorphism of $G/H$.
			The following statements hold true:
			\begin{enumerate}
				
				\item\label{MF:prop231} projection $p\colon G \to G/H$ induces epimorphism of each Reidemeister class of $\varphi$
					onto some Reidemeister class of $\overline{\varphi}$, in particular,~${R(\overline{\varphi})\leq R(\varphi)}$;
					
				\item $R(\varphi')\leq R(\varphi)\cdot |C(\overline{\varphi})|$;
					\textit{(see~\cite{TBFTDG}).}
				
				\item if $G/H$ is a finitely generated residually finite group, then~${R(\varphi) < \infty}$ holds if and only if ${R(\overline{\varphi})<\infty}$ and  ${ R(\tau_g\circ\varphi') < \infty }$ for all ${g\in G}$ \textit{(see~\cite{TIR})}.
				
			\end{enumerate}
		\end{proposition}
		
		Now we introduce the definition of restricted wreath product.
		
		\begin{definition}
			The \emph{restricted wreath product} of groups $G$ and $H$ is defined in the following way:
			\[
				G \WR H := \bigoplus_{x\in H} (G)_x \rtimes_\alpha H,\enspace\text{where } \alpha(h)(G)_x = (G)_{hx}.
			\]
			Note, that usage of the direct sum notation means that only finite sums of elements of~$G$ are allowed.
			The same notion with the direct product instead of the direct sum is known as \emph{unrestricted wreath product}.
		\end{definition}
		
		We consider the case of
		\[
			\Gamma_n = \bbz_n \WR \bbz^k = \bigoplus_{x\in \bbz^k}(\bbz_n)_x \rtimes_\alpha \bbz^k
		\]
		Here the subgroup~${\Sigma_n \mathrel{:=} \bigoplus(\bbz_n)_x}$ is completely invariant as the torsion subgroup.
		Hence, every automorphism~$\varphi$ of $\Gamma$ can be restricted on~$\Sigma$: ${\varphi'=\varphi\big\rvert_\Sigma}$,
			and the quotient automorphism~${\overline{\varphi}\colon\bbz^k\to\bbz^k}$ can be defined.

%		Though different automorphisms~$\varphi_1,\varphi_2$ of~$\Gamma$ can have same restrictions and same quotient automorphisms: ${\varphi_1'=\varphi_2'}$, ${ \overline{\varphi}_1 = \overline{\varphi}_2 }$, it will be shown later that for such automorphisms ${R(\varphi_1) = R(\varphi_2)}$, and in case of these numbers being finite, the TBFT\textsubscript{$f$} holds for both of them.
%		Considering the aforesaid fact, it is clear that the studied properties of an arbitrary~$\varphi$ depend only on the properties of~$\varphi'$ and~$\overline{\varphi}$.

	\subsection{Explicit form of automorphisms of \texorpdfstring{$\bbz_n\WR\bbz^k$}{}}\label{MF:secEFA}
	
		The statements of this section should be well-known for specialists, though the author did not manage to find an appropriate reference, so a short presentation with proofs will be given.

		Let $G$ be an inner semidirect product of its characteristic subgroup $N$ and subgroup $H$, i.~e.\  ${G=NH}$ and ${G\cap H = \{ e \}}$.
		It is well-known that each element ${g\in G}$ decomposes uniquely into product ${g=nh}$ with ${n\in N}$ and ${h\in H}$.
		
		For such $G$ and its automorphism $\varphi$ one can define a quotient automorphism ${\overline{\varphi} \colon H \to H}$ in the following way.
		First, define the projection ${ p \colon G \to G/N \cong H}$.
		Now put ${\overline{\varphi}(h) \mathrel{:=} p \circ \varphi(h) }$.

		\begin{definition}
			A mapping~$\psi$ from a group $H$ to a group $N$ is called a \emph{crossed homomorphism,} if there exists an action ${ \beta \colon H \to \Aut N }$ of $H$ on $N$,
				satisfying
			\begin{equation} \label{MF:chdef}
				\forall h,h'\in H\quad \psi(hh') = \psi(h)\bigl( \beta(h)\bigl(\psi(h')\bigr) \bigr).
			\end{equation}
		\end{definition}
		
		\begin{remark}
			Note that the mapping ${\psi \colon H \to \{ e \} \subset N}$ is a crossed homomorphism, regardless of $H$ and~$N$.
			Indeed, one can put ${\beta(h)(n)=n}$ for all ${h\in H}$ and ${n\in N}$.
		\end{remark}

		\begin{proposition}
			Let $\varphi$ be an automorphism of $G$, then for $g=nh$
			\begin{equation}\label{MF:tildephi}
				\varphi(g) = \varphi'(n)\widetilde{\varphi}(h)\overline{\varphi}(h),
			\end{equation}
			where $\varphi'$ is the restriction of $\varphi$ onto $N$, $\overline{\varphi}$ is the quotient automorphism, and ${\widetilde{\varphi}\colon H \to N}$ is some crossed homomorphism,
				satisfying~\eqref{MF:chdef} with the action
			\begin{equation}
				\beta(h)(n) \mathrel{:=} \overline{\varphi}(h) n \bigl( \overline{\varphi}(h) \bigr)^{-1}.
			\end{equation}
		\end{proposition}

		\begin{proof}[Proof]
			It is clear that ${\varphi(n)=\varphi'(n)}$.
			Let ${\varphi(h) = \widetilde{n}\widetilde{h}}$ for some ${\widetilde{n} \in N }$ and ${ \widetilde{h} \in H }$.
			Note that, on one hand, ${ p\circ \varphi(h) = p(\widetilde{n}\widetilde{h}) = \widetilde{h}}$.
			On the other hand, ${ p\circ \varphi(h) = \overline{\varphi}(h) \mathrel{=:} \overline{h}}$, so, ${ \widetilde{h} = \overline{h} }$.
			
			Define ${\widetilde{\varphi}(h) \mathrel{:=} \widetilde{n}}$.
			For ${ h,h'\in H }$,
			\[
				\varphi(hh') = \widetilde\varphi(hh')\overline{\varphi}(hh') =  \widetilde\varphi(hh')\overline{\varphi}(h)\overline{\varphi}(h').
			\]
			Also,
			\[
				\varphi(hh') = \varphi(h) \varphi(h') = \widetilde\varphi(h)\overline{\varphi}(h) \widetilde\varphi(h')\overline{\varphi}(h').
			\]
			Thus,
			\[
				\widetilde\varphi(hh')\overline{\varphi}(h) = \widetilde\varphi(h)\overline{\varphi}(h) \widetilde\varphi(h') \Longrightarrow
				\widetilde{\varphi}(hh') = \widetilde{\varphi}(h)\beta(h)\bigl( \widetilde{\varphi}(h') \bigr),
			\]
			where the action $\beta$ is defined by
			\[
				\beta(k)(m) = \alpha\bigl(\overline{\varphi}(k)\bigr)(m) = \overline{\varphi}(k)(m)\bigl(\overline{\varphi}(k)\bigr)^{-1},\quad k \in  H, m\in  N.
			\]
%			Mapping $\beta$ is indeed an action: observe that
%			\begin{multline*}
%				\beta(kl)(m) = \overline{\varphi}(kl)m\bigl(\overline{\varphi}(kl)\bigr)^{-1} = \overline{\varphi}(k)\overline{\varphi}(l)m\bigl(\overline{\varphi}(k)\overline{\varphi}(l)\bigr)^{-1} = {}\\[2mm]
%					{} = \overline{\varphi}(k)\overline{\varphi}(l)m\bigl(\overline{\varphi}(l)\bigr)^{-1}\bigl(\overline{\varphi}(k)\bigr)^{-1} = 
%					\overline{\varphi}(k)\beta(l)(m)\bigl(\overline{\varphi}(k)\bigr)^{-1} = 
%					\beta(k)\circ\beta(l)(m),
%			\end{multline*}
%			as well as
%			\begin{equation*}
%				\beta(e)(m) = \overline{\varphi}(e)m\bigl(\overline{\varphi}(e)\bigr)^{-1} = eme = m.
%			\end{equation*}
			
		\end{proof}

		Note that for an automorphism $\varphi$ the corresponding crossed homomorphism is unique.
		Indeed, if ${\varphi(nh) = \varphi'(n)\widetilde{\varphi}(h)\overline{\varphi}(h) = \varphi'(n)\widetilde{\phi}(h)\overline{\varphi}(h)}$ for all $n$ and $h$, then ${\widetilde{\phi} = \widetilde{\varphi}}$, since ${ \widetilde{\varphi}(h) = \widetilde{\phi}(h) = \varphi(h)\bigl( \overline{\varphi}(h) \bigr)^{-1}}$.
		Conversely, every automorphism $\varphi$ is completely described by its restriction, the quotient automorphism and the corresponding crossed homomorphism, since if ${\varphi' = \phi'}$, ${\overline\varphi = \overline\phi}$ and ${\widetilde\varphi = \widetilde\phi}$, then due to the formula~\eqref{MF:tildephi} ${\varphi = \phi}$.

		\begin{corollary}
			In our case of ${N=\Sigma}$ and ${H=\bbz^k}$, applying~\eqref{MF:tildephi}, one gets for ${\sigma\in\Sigma}$, ${z\in\bbz^k}$
			\begin{equation}\label{MF:phi}
				\varphi\big((\sigma,z)\big) = \big( \varphi'(\sigma)+\widetilde{\varphi}(z),\overline{\varphi}(z) \big).
			\end{equation}
		\end{corollary}

		\begin{proposition}\label{MF:phiexists}
			Let~${\Phi'\in\Aut\Sigma}$,~${\overline{\Phi}\in\Aut \bbz^k}$ be given, as well as some crossed homomorphism ${\widetilde\Phi \colon \bbz^k \to \Sigma}$.
			An automorphism $\varphi$ of ${\Gamma=\Sigma\rtimes_\alpha\bbz^k}$ with the corresponding
				${\varphi'=\Phi'}$, ${\overline{\varphi}=\overline{\Phi}}$, ${\widetilde{\varphi}=\widetilde{\Phi}}$ exists, if and only if
			\begin{align}
				\forall\sigma\in \Sigma\enspace \forall z\in Z\quad
				 &\Phi'( \alpha(z)(\sigma)) = \alpha(\overline{\Phi}(z))( \Phi'(\sigma)), \label{MF:eqconn1}\\
				 \forall z_1, z_2 \in \bbz^k \quad 
				 &\widetilde{\Phi}(z_1+z_2) = 
	 				\widetilde{\Phi}(z_1) +
	 				 	\alpha\bigl(
	 				 		\overline{\Phi}(z_1)
	 				 	\bigr)\bigl(
	 				 		\widetilde{\Phi}(z_2)
	 				 	\bigr). \label{MF:eqconn2}
			\end{align}
			Moreover, if it exists, then it satisfies~\eqref{MF:phi}.
		\end{proposition}

		\begin{proof}[Proof]
			Define ${\varphi}$ using the formula ${\varphi\bigl( \sigma,z \bigr) = \bigl( \Phi'(\sigma)+\widetilde{\Phi}(z), \overline{\Phi}(z) \bigr)}$.
			It is clear that in this case ${\varphi'=\Phi'}$, ${\overline{\varphi}=\overline{\Phi}}$,~${\widetilde{\varphi}=\widetilde{\Phi}}$.
			
			Let~${\sigma_1,\sigma_2\in\Sigma}$,~${z_1,z_2\in \bbz^k}$, then
			\begin{multline}\label{MF:eqPHI1}
				\varphi\bigl( (\sigma_1,z_1) \cdot (\sigma_2,z_2) \bigr) = 
				\varphi\bigl( (  \sigma_1 + \alpha(z_1)(\sigma_2), z_1 + z_2  ) \bigr) = \\[2mm]
			 = \Bigl(
				 	\Phi'(\sigma_1) +
				 	\Phi'\bigl(
				 		\alpha(z_1)(\sigma_2)
				 	\bigl) +
				 	\widetilde{\Phi}(z_1+z_2),
				 	\overline{\Phi}(z_1+z_2)
			 \Bigr),
			\end{multline}
			as well as
			\begin{multline}\label{MF:eqPHI2}
				\varphi\bigl(
					(\sigma_1,z_1)
				\bigr) \cdot
				\varphi\bigl(
					(\sigma_2,z_2)
				\bigr) = 
				\bigl(
					\Phi'(\sigma_1) + \widetilde{\Phi}(z_1), \overline{\Phi}(z_1)
				\bigr)\cdot
				\bigl(
					\Phi'(\sigma_2)+\widetilde{\Phi}(z_2), \overline{\Phi}(z_2)
				\bigr) = \\[2mm]
			 = \Bigl(
				 	\Phi'(\sigma_1) +
				 	\alpha\bigl(
				 		\overline{\Phi}(z_1)\bigr)\bigl(\Phi'(\sigma_2)
				 	\bigr) + 
				 	\widetilde{\Phi}(z_1) +
				 	\alpha\bigl(
				 		\overline{\Phi}(z_1)
				 	\bigr)\bigl(
				 		\widetilde{\Phi}(z_2)
				 	\bigr),
				 	\overline{\Phi}(z_1) + \overline{\Phi}(z_2)
			 \Bigr).
			\end{multline}
			Note that ${\overline{\Phi}(z_1+z_2) = \overline{\Phi}(z_1) + \overline{\Phi}(z_2)}$.
			The mapping $\varphi$ is a homomorphism, if and only if the RHS of \eqref{MF:eqPHI1} equals the RHS of \eqref{MF:eqPHI2}, or
			\begin{multline}\label{MF:eqPHI3}
				\Phi'(\sigma_1) +
				 	\Phi'\bigl(
				 		\alpha(z_1)(\sigma_2)
				 	\bigl) +
				 	\widetilde{\Phi}(z_1+z_2)
				 	= {} \\
				 	{} = 
				 \Phi'(\sigma_1) +
				 	\alpha\bigl(
				 		\overline{\Phi}(z_1)\bigr)\bigl(\Phi'(\sigma_2)
				 	\bigr) + 
				 	\widetilde{\Phi}(z_1) +
				 	\alpha\bigl(
				 		\overline{\Phi}(z_1)
				 	\bigr)\bigl(
				 		\widetilde{\Phi}(z_2)
				 	\bigr).
			\end{multline}
			Putting ${\sigma_2 = 0}$ yields
			\[
				\widetilde{\Phi}(z_1+z_2) = \widetilde{\Phi}(z_1) + \alpha\bigl(\overline{\Phi}(z_1)\bigr)\bigl(\widetilde{\Phi}(z_2)\bigr),
			\]
			which coincides with~\eqref{MF:eqconn2}.
			Combining it with~\eqref{MF:eqPHI3} yields
			\[
				\Phi'\bigl(\alpha(z_1)(\sigma_2)\bigl) = \alpha\bigl(\overline{\Phi}(z_1)\bigr)\bigl(\Phi'(\sigma_2)\bigr),
			\]
			which is, essentially,~\eqref{MF:eqconn1}.
			
			Now we shall verify that $\varphi$ is monomorphic.
			If $\varphi(\sigma,z) = \big( \Phi'(\sigma)+\widetilde{\Phi}(z),\overline{\Phi}(z) \big) = (0,0)$, then ${\overline{\Phi}(z)=0}$, so ${z=0}$, since ${\overline{\Phi}}$ is monomorphic.
			Therefore ${\widetilde{\Phi}(z) = 0}$, so ${\Phi'(\sigma) = 0}$, thus ${\sigma = 0}$, since $\Phi'$ is monomorphic.
			
			To complete the proof we show that $\varphi$ is an epimorphism.
			For an arbitrary ${\sigma\in \Sigma}$ and ${z\in \mathbb{Z}^k}$ we construct such ${\sigma_0\in \Sigma}$ and ${z_0\in \mathbb{Z}^k}$ that ${\varphi(\sigma_0,z_0) = (\sigma,z)}$.
			Since $\overline{\Phi}$ is epimorphic, there exists such $z_0$ that ${\overline\Phi(z_0)=z}$.
			Since $\Phi'$ is epimorphic, there exists such $\sigma_0$ that ${\Phi'(\sigma_0) = \sigma - \widetilde{\Phi}(z_0)}$.
			Observe that
			\[
				\varphi(\sigma_0,z_0) =
				\bigl( \Phi'(\sigma_0)+\widetilde{\Phi}(z_0), \overline{\Phi}(z_0) \bigr) =
				\bigl( \sigma - \widetilde{\Phi}(z_0) + \widetilde{\Phi}(z_0), z \bigr) = (\sigma,z),
			\]
			so the element ${ (\sigma_0, z_0) }$ is the desired one.
		\end{proof}
		
		\begin{remark}
			Note, that if $\overline{\Phi}$ and $\Phi'$ satisfy~\eqref{MF:eqconn1}, then there exists $\varphi$ with
				${\varphi'=\Phi'}$, ${\overline{\varphi}=\overline{\Phi}}$ and ${\widetilde{\varphi} \equiv 0}$.
		\end{remark}

	\section{Projection Theorem}\label{MF:secPT}

		Let $s$, $n$, and $k$ be positive integers.
		Define groups~$G$ and $\Gamma$ as ${\bbz_n\WR\bbz^k}$ and ${\bbz_{sn}\WR\bbz^k}$ respectively.
		By~$\Omega$ and~$\Sigma$ we denote their respective torsion subgroups.
		
		Now we define a projection~$\Pi$ from~$\Gamma$ onto~$G$.
		For ${x\in\bbz^k}$ denote by $\Delta_x$ and $\delta_x$ generators of groups ${ (\bbz_{sn})_x }$ and ${ (\bbz_{n})_x }$ respectively.
		First, define the projection ${\pi\colon\Sigma\to\Omega}$ in the following way:
		\begin{equation}\label{MF:PiDef}
			\pi\left( k_1\Delta_{x_1} + \dots + k_l\Delta_{x_l} \right) = (k_1 \mathbin{\mathrm{mod}} n)\delta_{x_1} + \dots + (k_l \mathbin{\mathrm{mod}} n)\delta_{x_l},
		\end{equation}
		where ${k_1,\dots,k_l\in\bbz_{sn}}$, ${x_1,\dots,x_l\in\bbz^k}$.
		The projection~$\Pi$ can now be defined as
		\[
			\Pi(\sigma,z) = (\pi(\sigma),z),
		\]
		where~${\sigma\in\Sigma}$, ${z\in\bbz^k}$. Observe that $\Pi$ is a homomorphism, since
		\begin{multline*}
			\Pi\Biggl( \Bigl( \sum_{i=1}^{l} k_i\Delta_{x_i}, z \Bigr) \cdot \Bigl( \sum_{j=1}^{r} m_j\Delta_{y_j}, w \Bigr) \Biggr) = 
			\Pi\Biggl( \Bigl( \sum_{i=1}^{l} k_i\Delta_{x_i} + \sum_{j=1}^{r} m_j\Delta_{y_j+z}, z + w \Bigr) \Biggr) = {} \\
				\Bigl( \sum_{i=1}^{l} (k_i \mathbin{\mathrm{mod}} n)\delta_{x_i} + \sum_{j=1}^{r}(m_j \mathbin{\mathrm{mod}} n)\delta_{y_j+z}, z + w \Bigr) = {}\\
				\Bigl( \sum_{i=1}^{l} (k_i \mathbin{\mathrm{mod}} n)\delta_{x_i}, z \Bigr) \cdot \Bigl( \sum_{j=1}^{r}(m_j \mathbin{\mathrm{mod}} n)\delta_{y_j}, w \Bigr) = {}\\
				\Pi\Biggl( \Bigl( \sum_{i=1}^{l} k_i\Delta_{x_i}, z \Bigr) \Biggr)\cdot \Biggl(\Bigl( \sum_{j=1}^{r} m_j\Delta_{y_j}, w \Bigr) \Biggr).
		\end{multline*}
		
		\begin{lemma}
			The subgroup $\ker \Pi$ is completely invariant in $\Gamma$.
		\end{lemma}
		
		\begin{proof}[Proof]
			Let ${\pi(\sigma)=0\in\Omega}$, where~${\sigma = \sum k_i\Delta_{x_i}}$.
			Considering~\eqref{MF:PiDef}, one can observe that all coefficients~$k_i$ are divisible by~$n$.
			Hence, ${\sigma=n\cdot\widetilde{\sigma}}$ for some ${\widetilde{\sigma}\in\Sigma}$.
			If $\chi$ is an endomorphism of~$\Sigma$, then 
			\[
				\pi\circ\chi(\sigma) = \pi\circ\chi(n\cdot\widetilde{\sigma}) = \pi (n \cdot \chi(\widetilde{\sigma})) = 0,
			\]
			so $\ker\pi$ is a completely invariant subgroup of~$\Sigma$, which is completely invariant in~$\Gamma$, thus 
				$\ker\Pi$ is completely invariant in $\Gamma$.
		\end{proof}
	
		Now we can state the main result of this section.
		
		\begin{theorem}\label{MF:theor1}
			Let $\varphi$ be an automorphism of~$\Gamma$.
			Then there exists such an automorphism~$\psi$ of~$G$, that the following diagrams commute.
			\begin{equation}\label{MF:commdiag}
				\begin{CD}
				\Gamma @>\varphi>> \Gamma\\
				@V\Pi VV @VV\Pi V\\
				G @>\psi>> G
			\end{CD}
			\qquad\qquad
			\begin{CD}
				\Sigma @>\varphi'>> \Sigma\\
				@V\pi VV @VV\pi V\\
				\Omega @>\psi'>> \Omega
			\end{CD}
			\qquad\qquad
			\begin{CD}
				\bbz^k @>\overline{\varphi}>> \bbz^k\\
				@V\mathrm{Id} VV @VV\mathrm{Id} V\\
				\bbz^k @>\overline{\psi}>> \bbz^k
			\end{CD}
			\end{equation}
			Moreover, ${R(\varphi)\geq R(\psi)}$.
		\end{theorem}
	
		\begin{proof}[Proof]
			Let $g$ be an element of $G$.
			Define ${\psi(g)}$ in the following way:
				let ${\gamma\in\Gamma}$ satisfy ${\Pi(\gamma) = g}$;
				put ${\psi(g) = \Pi\circ\varphi(\gamma)}$.
			This way $\psi$ is well defined, since $\ker \Pi$ is characteristic.
			
			By Proposition~\ref{MF:prop23}~(\ref{MF:prop231}), ${R(\varphi)\geq R(\psi)}$.
		\end{proof}
		
		\begin{corollary}\label{MF:cor}
			The groups ${\bbz_{2n} \WR \bbz^k}$ and ${\bbz_{3n}\WR\bbz^{2k+1}}$ have the $R_\infty$-property.
		\end{corollary}
		
		\begin{proof}[Proof]
			In~\cite{TR} it was shown that the groups ${\bbz_{2} \WR \bbz^k}$ and ${\bbz_{3}\WR\bbz^{2k+1}}$ have the $R_\infty$-property.
			Now apply Theorem~\ref{MF:theor1}.
		\end{proof}
	
	\section{Groups without \texorpdfstring{$R_\infty$}{}-property}\label{MF:secRI}
		
%		In this section we prove that for all ${ \bbz_n\WR\bbz^k }$, other than stated in Corollary~\ref{MF:cor}, there exists an automorphism $\varphi$ with finite Reidemeister number, so these groups do not have $R_\infty$-property.

		In~\cite{TR} the following two results were proved.
		\begin{proposition}\label{MF:trfone}
			If an automorphism $\varphi$ of the group ${\bbz_p\WR\bbz^k}$ for prime $p$ has finite Reidemeister number,
				then ${R(\overline{\varphi})<\infty}$ and ${R(\tau_\gamma\circ\varphi')=1}$ for all ${\gamma\in\Gamma}$.
		\end{proposition}
		
		\begin{proposition}\label{MF:trres}
			If an automorphism $\varphi$ of the group ${\bbz_p\WR\bbz^k}$ for prime $p$ satisfies ${R(\varphi')=1}$, then ${R(\tau_\gamma\circ\varphi')=1}$ for all ${\gamma\in\Gamma}$.
		\end{proposition}
		
		Now we generalize these results.
		
		\begin{theorem}\label{MF:theor2}
			Suppose  $n,m,k$ are positive integers and $\varphi$ is an automorphism of ${\bbz_{nm}\WR\bbz^k}$,
				such that the automorphisms $\varphi_n$ and $\varphi_m$, defined by the commutative diagrams
			\begin{equation}
				\begin{CD}
				\bbz_{nm}\WR\bbz^k @>\varphi>> \bbz_{nm}\WR\bbz^k\\
				@V\Pi_n VV @VV\Pi_n V\\
				\bbz_{n}\WR\bbz^k @>\varphi_n>> \bbz_{n}\WR\bbz^k
			\end{CD}
			\qquad\qquad
			\begin{CD}
				\bbz_{nm}\WR\bbz^k @>\varphi>> \bbz_{nm}\WR\bbz^k\\
				@V\Pi_m VV @VV\Pi_m V\\
				\bbz_{m}\WR\bbz^k @>\varphi_m>> \bbz_{m}\WR\bbz^k
			\end{CD},
			\end{equation}
			satisfy ${R(\varphi'_n)=R(\varphi'_m)=1}$.
			Then ${R(\varphi')=1}$.
		\end{theorem}
		
		\begin{remark}
			In the case of ${n=m=p}$, where $p$ is prime, this theorem, combined with Theorem~\ref{MF:theor1}, implies that ${R(\varphi')=1}$, if and only if ${R(\varphi_{p}')=1}$, where
				${\varphi}$ is an automorphism of ${\Aut \bbz_{p^s} \WR \bbz^k}$ and ${\Pi_{p} \circ \varphi = \varphi_p \circ \Pi_p}$.
		\end{remark}
		
		\begin{proof}[Proof]
			Note that the condition ${R(\varphi')=1}$ is equivalent to the condition of ${\chi = 1-\varphi'}$ to be an epimorphism.
			Denote by $\chi_n$ and $\chi_m$ the epimorphisms ${1-\varphi'_n}$ and ${1-\varphi'_m}$ respectively.
			As before, $\Sigma_n$ and $\Sigma_m$ denote the torsion subgroups of ${\bbz_{n}\WR\bbz^k}$ and ${\bbz_{m}\WR\bbz^k}$, while $\Sigma$ is the torsion subgroup of $\Gamma$.

			Consider the non-homomorphic embedding  ${\iota_n \colon \Sigma_n \hookrightarrow \Sigma}$, induced by the embedding
			${\bbz_n \cong \{ 1,\dots,n\} \stackrel{\Id}{\hookrightarrow} \{ 1,\dots,sn \} \cong \bbz_{sn}  }$,
			 and,  in the same manner, ${\iota_m \colon \Sigma_m \hookrightarrow \Sigma}$.
			Note, that ${\pi_n \circ \iota_n = \Id }$ and ${\pi_m \circ \iota_m = \Id }$.

			To prove that $\chi$ is an epimorphism, we show that for an arbitrary ${z\in Z}$ there exists ${\sigma\in\Sigma}$, such that ${\chi(\sigma)=\Delta_z}$ ($\Delta_z$~is the generator of~$(\bbz_{nm})_z$).

			Let ${\eta_1\in\Sigma_n}$ satisfy ${\chi_n(\eta_1)=\Delta_z}$, then ${\chi\circ\iota_n(\eta_1) = \Delta_z + n\cdot\theta}$ for some ${\theta\in\Sigma}$.
			Indeed, ${\pi_n \circ \chi \circ \iota_n(\eta_1) = \chi_{n} \circ \pi_n \circ \iota_n (\eta_1) = \chi_n(\eta_1) = \Delta_z}$.
			Now consider ${\eta_2\in\Sigma_m}$, such that ${ \chi_m(\eta_2) = \pi_m(\theta) }$.
			In this case ${ \chi\circ \iota_m(\eta_2) = \theta + m\kappa }$ for some ${\kappa\in\Sigma}$.
			
			Observe that
			\begin{equation*}
				\chi\bigl(\iota_n(\eta_1) - n\cdot\iota_m(\eta_2)\bigr) = \Delta_z + n\theta - n(\theta + m\kappa) =
				\Delta_z + n\theta - n\theta + nm\kappa = \Delta_z.
			\end{equation*}
			Put ${\sigma = \iota_n(\eta_1) - n\cdot\iota_m(\eta_2) }$.
		\end{proof}

		\begin{corollary}
			If for an automorphism $\varphi$ of ${\Gamma = \bbz_{p_1^{s_1}\dots p_N^{s_n}} \WR \bbz^k }$ it holds true that ${R(\varphi')=1}$, then ${R(\tau_\gamma\circ\varphi')=1}$ for all ${\gamma\in\Gamma}$.
		\end{corollary}
		
		\begin{proof}[Proof]
			If ${R(\varphi')=1}$, then ${R(\varphi'_{p_i})=1}$ for all $i$, since if ${1-\varphi'}$ is an epimorphism, then so is ${1-\varphi'_{p_i}}$, since
				$\pi_{p_i}$ is an epimorphism and ${(1-\varphi'_{p_i}) \circ \pi_{p_i} = \pi \circ (1-\phi')}$.
			Proposition~\ref{MF:trres} implies that ${R\bigl(\tau_{\Pi(\gamma)}\circ \varphi'_{p_i}\bigr)=1}$ for all ${\gamma\in\Gamma}$.
			If ${\pi_{p_i}\circ\varphi'=\varphi'_{p_i}\circ\Pi_{p_i}}$, then ${\pi_{p_i}\circ \tau_\gamma\circ\varphi'=\tau_{\Pi_{p_i}(\gamma)}\circ \varphi'_{p_i}\circ\Pi_{p_i}}$.
			Thus, by Theorem~\ref{MF:theor2}, ${R(\tau_\gamma\circ\varphi')}$ equals one for all ${\gamma\in\Gamma}$.
		\end{proof}
		
%		Theorems~\ref{MF:theor1} and~\ref{MF:theor2} show, that for ${\varphi\in\Aut \bbz_{p_1^{s_1}\dots p_N^{s_n}} \WR \bbz^k }$ the number $R(\varphi')$ is finite,
%			if and only if all of the numbers ${R(\varphi_{p_i})}$ are finite,
%			where $\varphi_{p_i}$ are automorphisms of ${\bbz_{p_i}\WR\bbz^k}$, satisfying ${\varphi_{p_i}\circ\Pi_{p_i} = \Pi_{p_i}\circ\varphi}$.
%		Proposition~\ref{MF:trres} shows that ${R(\varphi_{p_i})}$ is finite, if and only if it is equal to~1.
%		

		Note, that for ${\varphi \in \Aut \bbz_{p_1^{s_1}\dots p_N^{s_n}} \WR \bbz^k }$ the number ${R(\varphi')}$ is either equal to one, or infinite.
		Indeed, if it is finite, then, by Theorem~\ref{MF:theor1}, all $R(\varphi_{p_i})$ are finite, so, by Proposition~\ref{MF:trfone} all $R(\varphi'_{p_i})$ equal one, thus, by Theorem~\ref{MF:theor2}, ${R(\varphi')=1}$.
		
		Corollary~\ref{MF:cor} shows that the groups ${\bbz_{2n} \WR \bbz^k}$ and ${\bbz_{3n}\WR\bbz^{2k+1}}$ have the $R_\infty$-property, due to the fact that, for the former, the number $R(\varphi_2)$ is always infinite and for the latter, the number $R(\varphi_3)$ is always infinite.
		It should be noted, that the fact that all of $\bbz_{p_i} \WR \bbz^k$ admit automorphisms with finite Reidemeister does not automatically imply that ${\bbz_{p_1^{s_1}\dots p_N^{s_n}} \WR \bbz^k }$ does.
		However, Propositions~\ref{MF:prop3} and~\ref{MF:prop2d} show, that this implication holds true and is somewhat converse to Corollary~\ref{MF:cor}.
		
		Due to Proposition~\ref{MF:prop23}, to prove that the group ${\bbz_{p_1^{s_1}\dots p_N^{s_n}} \WR \bbz^k }$ admits an automorphism~$\varphi$ with finite Reidemeister number it is sufficient to verify
			that there exists $\varphi$ with ${R(\varphi')=1}$ and ${R(\overline{\phi})< \infty}$, where the first condition is equivalent to ${R(\varphi_{p_i^{s_i}})=1}$ for all~$i$.

		\begin{proposition}\label{MF:prop3}
			The group ${\bbz_{p_1^{s_1}\dots p_N^{s_n}}\WR\bbz^k}$, where all ${p_i>3}$, admits an automorphism with finite Reidemeister number.
		\end{proposition}
		
		\begin{proof}[Proof]
			Take the automorphism $\varphi$, defined by~\eqref{MF:phi} with
			\[
				\varphi'(\Delta_z) = 2\Delta_{-z},\quad \overline{\varphi} = -\Id,\quad \widetilde{\varphi}=0.
			\]
			Note, that Proposition~\ref{MF:phiexists} provides its existence.
			For every~$i$ the automorphism $\varphi_{p_i}$ satisfies
			\[
				\varphi_{p_i}'(\delta_z) = 2\delta_{-z},\quad \overline{\varphi}_{p_i} = -\Id,\quad \widetilde{\varphi}_{p_i}=0.
			\]
			Note that ${R(\overline{\varphi})=2^k<\infty}$ and in~\cite{TR} it was shown that ${1-\varphi_{p_i}}$ is an epimorphism, so apply Theorem~\ref{MF:theor2} to complete the proof.
		\end{proof}

		\begin{proposition}\label{MF:prop2d}
			The group ${\bbz_{p_1^{s_1}\dots p_N^{s_n}}\WR\bbz^{2k}}$, where all ${p_i>2}$, admits an automorphism with finite Reidemeister number.
		\end{proposition}

		\begin{proof}[Proof]
			Take the automorphism $\varphi$, defined by~\eqref{MF:phi} with
			\[
				\varphi'(\Delta_z) = m\Delta_{\overline{\varphi}(z)}, \quad \widetilde{\varphi}=0,\quad \overline{\varphi},
			\]
			where $\overline{\varphi}$ has the matrix $\mathbf{M}$ in standard basis in $\bbz^{2k}$, where
			\[
				\mathbf{M} = \underbrace{M \oplus \dots \oplus M}_k,\quad
								M = \begin{pmatrix}
									0  & 1  \\
									-1 & -1
								\end{pmatrix}.
			\]
			The value of $m$ will be chosen later.
			Note, that Proposition~\ref{MF:phiexists} provides the existence of $\varphi$ regardless of~$m$.
			Observe that $\mathbf{M}^3=E$.
			For~${p_i\neq7}$ consider
			\[
				\varphi'_{{p_i}^{s_i}}(\delta_z) = 2\delta_{\mathbf{M}z}.
			\]
			It is easy to verify that for $z\in\bbz^k$
			\[
				\left(1 - \varphi'_{{p_i}^{s_i}}\right) \left( -\frac{ \delta_{z} + 2\delta_{\mathbf{M}z} + 4\delta_{\mathbf{M^2}z}}{7} \right) = 
				\delta_z,
			\]
			so, ${1 - \varphi'_{{p_i}^{s_i}}}$ is an epimorphism.
			If ${p_i\neq7}$ for all~$i$, put $m=2$.
			
			If ${p_{i^*} = 7}$ put
			\[
				\varphi_{7\strut^{s_{i^*}}}'(\delta_z) = 3\delta_{\mathbf{M}z},
			\]
			then
			\[
				\left(1 - \varphi'_{{7}\strut^{s_{i^*}}}\right) \left( -\frac{ \delta_{z} + 2\delta_{\mathbf{M}z} + 4\delta_{\mathbf{M^2}z}}{26} \right) = 
				\delta_z.
			\]
			Now $m$ can be chosen as a solution of the system
			\[
				\begin{cases}
					m \equiv 3 \pmod{7\strut^{s_{i^*}}},\\
					m \equiv 2 \pmod{{p_i}^{s_i}}, \text{ for } i \neq i^*.
				\end{cases}
			\]
			The solution exists due to the Chinese remainder theorem.
			
			Observe that for such $m$ holds ${\varphi'_{p_i^{s_i}}(\delta_z) = 2\delta_{\overline{\varphi}(z)}}$
			for ${i \neq i^*}$ and
			${\varphi'_{7\strut^{s_{i^*}}}(\delta_z) = 3\delta_{\overline{\varphi}(z)}}$,
			since ${\varphi'(\Delta_z) = m\Delta_{\overline{\varphi}(z)}}$
			and ${\Pi\circ\varphi' = \varphi'_{{p_i}^{s_i}} \circ \Pi_{p_i^{s_i}}}$.
		\end{proof}

	\section{TBFT\texorpdfstring{\textsubscript{$f$}}{} in cases of finite \texorpdfstring{$R(\varphi)$}{}}\label{MF:secTBFT}

		For a group~$G$ one can define the \emph{unitary dual} $\widehat{G}$, i.~e. the set of equivalence classes of unitary irreducible representations of~$G$.
		We, however, are going to consider only ${\widehat{G}_f\subset\widehat{G}}$, which is the subset of all finite-dimensionals representations.
		Each automorphism $\varphi$ of $G$ induces a map ${\widehat{\varphi} \colon \widehat{G}_f \to \widehat{G}_f}$, ${\widehat{\varphi}(\rho) = \rho \circ \varphi}$.
		%Now we prove that for all automorphisms with finite Reidemeister number the TBFT\textsubscript{$f$} holds.
		
		\begin{proposition}\label{MF:tbftfor}
			Let $\varphi$ be an automorphism of ${ \Gamma = \bbz_n\WR\bbz^k }$ with ${R(\overline{\varphi})<\infty}$ and ${R(\tau_\gamma\circ\varphi')=1}$ for all ${\gamma\in\Gamma}$.
			Then the TBFT\textsubscript{$f$} holds for $\varphi$, namely ${R(\varphi)=\#\mathrm{Fix}(\widehat{\varphi})}$.
		\end{proposition}
		
		\begin{proof}[Proof]
			First, note that ${R(\varphi)\geq\#\mathrm{Fix}(\widehat{\varphi})}$.
			A proof can be found in~\cite{TBFTDG}.
			
			To prove the inverse inequality, consider the projection ${P\colon \bbz_n\WR\bbz^k\to\bbz^k}$.
			Since ${R(\varphi')=1}$, only one $\varphi$-class is mapped onto the $\overline{\varphi}$-class of the unit element.
			And, since ${R(\tau_\gamma\circ\varphi')=1}$ for all ${\gamma\in\Gamma}$, for each $\overline{\varphi}$-class in $\bbz^k$
				only one $\varphi$-class is mapped onto it, so $P$ induces a bijection of Reidemeister classes of $\varphi$
					and Reidemeister classes of $\overline{\varphi}$, i.~e. ${R(\varphi)=R(\overline{\varphi})}$.
			
			Since TBFT\textsubscript{$f$} holds for finitely-generated abelian groups (see~\cite{DZF}), $R(\varphi)=R(\overline{\varphi})=\#\mathrm{Fix}(\widehat{\overline{\varphi}})$.
			Note that ${\#\mathrm{Fix}(\widehat{\overline{\varphi}})\leq\#\mathrm{Fix}(\widehat{\varphi})}$, for if $\rho$ is a $\widehat{\overline{\varphi}}$-fixed representation, then
				 ${\rho\circ p}$ is a $\widehat{\varphi}$-fixed representation, and if $\rho_1$ and $\rho_2$ are not equivalent, then
				 ${\rho_1\circ p}$ and ${\rho_2\circ p}$ are not equivalent as well, which follows from he surjectivity of~$P$.
		\end{proof}

		\begin{theorem}\label{MF:theortbft}
			For every positive integer $k$ the finite-dimensional twisted Burnside--Frobenius theorem holds for the groups ${\bbz_{3n} \WR \bbz^{2k}}$, ${\gcd(n,2)=1}$ and ${\bbz_{n}\WR\bbz^k}$, ${\gcd(n,6)=1}$, i.~e. in all cases, where $R(\varphi)$ can be finite.
		\end{theorem}

		\begin{proof}[Proof]
			Apply Propositions \ref{MF:prop3},~\ref{MF:prop2d}, and~\ref{MF:tbftfor}.
		\end{proof}

	\vspace*{2\baselineskip}

	{
	\scriptsize\scshape\par
	
	\noindent
	Mikhail Igorevich Fraiman\\
	Faculty of Mechanics and Mathematics,\\
	Lomonosov Moscow State University,\\
	1, Leninskiye Gory st.,\\
	119991, Moscow, Russia,\\[2mm]
	Moscow Center for Fundamental and Applied Mathematics,\\
	MSU Department,\\
	1, Leninskiye Gory st.,\\
	119991, Moscow, Russia,\par}
	
	\noindent
	\texttt{mfraiman@abc.math.msu.su}
	
\end{document}